\documentclass[12pt]{article}
\usepackage{amsmath,amsthm,amssymb,latexsym,amsfonts,epsfig, color, geometry, enumerate,authblk}
\usepackage{hyperref}

\newcommand{\ip}[2]{\langle #1 , #2 \rangle}    

\newcommand{\Stab}{\operatorname{Stab}}

\newcommand{\B}{{\mathcal B}}

\newcommand{\id}{\operatorname{id}}

\newcommand{\smat}[1]{\left[\begin{smallmatrix}
 #1 \end{smallmatrix}\right]}

\newcommand{\cK}{\operatorname{\mathcal{K}}}
\newcommand{\cR}{\operatorname{\mathcal{R}}}
\newcommand{\cS}{\operatorname{\mathcal{S}}}

\newcommand{\rank}{\operatorname{rank}}

\newcommand{\Sym}{\mathbb S}

\newcommand\N{{\mathbb N}}
\newcommand\R{{\mathbb R}}

\newcommand{\tr}{^{\top}}

\newtheorem{theorem}{Theorem}
\newtheorem{claim}{Claim}
\newtheorem{lemma}{Lemma}
\newtheorem{proposition}{Proposition}
\newtheorem{definition}{Definition}
\newtheorem{remark}{Remark}

\newcommand{\interior}{\operatorname{int}}

\newcommand{\lrp}[1]{\left ( {#1} \right )}

\newcommand{\lrc}[1]{\left \{ {#1} \right \}}

\newcommand{\orb}{{\operatorname{\mathcal{O}}}}
\newcommand{\cyl}[2]{_{#1\times} {#2}}

\usepackage[numbers,sort]{natbib}
 \bibpunct[, ]{[}{]}{,}{n}{}{,}%

\author{  Olga Kuryatnikova  \thanks{ \href{mailto:o.kuryatnikova@uvt.nl}{o.kuryatnikova@uvt.nl} (corresponding author),
ORCID:  0000-0001-8460-7296} \quad and Juan C. Vera \thanks{\href{mailto:j.c.veralizcano@uvt.nl}{j.c.veralizcano@uvt.nl}}
 \and
\small{ Department of Econometrics and Operations Research, Tilburg University, 5037 AB Tilburg, Netherlands}
}

\title{Generalizations of Schoenberg's theorem on positive definite kernels}

\begin{document}

\maketitle

\begin{abstract}
The seminal theorem of I.J. Schoenberg characterizes positive definite (p.d.) kernels on the unit sphere $S^{n-1}$ invariant under the automorphisms of the sphere. We obtain two generalizations of this theorem for p.d. kernels on fiber bundles. Our first theorem characterizes invariant p.d. kernels on bundles whose fiber is a product of a compact set and the unit sphere. This result implies, in particular, a characterization of invariant under the automorphisms of the sphere  p.d. kernels on a product of $S^{n-1}$ and a compact set. Our second result characterizes invariant p.d. kernels on the  bundle  whose fiber is $S^{n-1}$, base space is $(S^{n-1})^{r}$ and map is the projection on the base space. This set of kernels is isomorphic to the set  of invariant under the automorphisms of the sphere  continuous functions $F$ on $(S^{n-1})^{r+2}$ such that $F(\cdot,\cdot ,Z)$ is positive definite for every $Z \in (S^{n-1})^{r}$. When $Z$ is fixed, this class reduces to the class of p.d. kernels invariant under the stabilizer of $Z$ in the automorphism group of the sphere. For $r=1$ these kernels have been used to obtain upper bounds for the spherical codes problem. Our extension for $r>1$ can be used to construct new upper bounds on the size of spherical codes. \\ 

\noindent Keywords: Schoenberg's theorem; positive definite kernels, spherical codes

\end{abstract}

\section{Introduction} \label{sec:introKissing2}
The seminal theorem of I.J. Schoenberg~\cite{Schoenberg} characterizes positive  definite  kernels on the unit sphere $S^{n-1}$ invariant under the automorphisms of the sphere. A kernel on $S^{n-1}$ is a continuous function $K:S^{n-1}{\times}S^{n-1}\rightarrow\R$ such that  $F(x,y){=}F(y,x)$ for all $x,y\in S^{n-1}$. A kernel is positive definite (p.d.) if for any finite $U \subset S^{n-1}$ the restriction of $K$ to $U\times U$ is a positive semi-definite matrix. Let $O_n$ be the orthogonal group in dimension $n$, which is the automorphism group of $S^{n-1}$. Schoenberg~\cite{Schoenberg} proved that $K$ is a positive definite kernel  on $S^{n-1}$ invariant under the action of $O_n$ if and only if
\begin{align*}
K(x,y)=\sum_{i  \in \N} c_i P^{\tfrac{n}{2}{-}1}_i  (x \tr y ), \ c_i \ge 0,
\end{align*}
where $P^{\tfrac{n}{2}{-}1}_i$ is a Gegenbauer polynomial  of order $\tfrac{n}{2}{-}1$ and degree $i$.

In this paper we provide two generalizations of the above result to positive definite kernels on fiber bundles. Theorem~\ref{thm:leftFixed} characterizes p.d. kernels on bundles whose fiber is a product of a compact set and the unit sphere. In particular,  for any compact $V\subseteq \R^r$, we obtain a characterization of p.d. kernels on $V \times S^{n-1}$ which are invariant under the action of $O_n$ when the arguments from $V$ are fixed. Next, Theorem~\ref{thm:PSD_multivKissing2}  characterizes the class of continuous functions $F(x,y,Z)$ on $(S^{n-1})^{r+2}$ invariant under $O_n$ such that $F(\cdot,\cdot ,Z)$ is a p.d. kernel for every $Z\in (S^{n-1})^r$. This class of functions can as well be viewed as kernels on a fiber bundle  whose fiber is $S^{n-1}$, base space is $(S^{n-1})^{r}$ and map is the projection on the base space.

Our work is inspired by the connection of p.d. kernels to some combinatorial problems on infinite compact graphs. The well known linear programming upper bound for the kissing number problem by~\citet{DGS} can be obtained using Schoenberg's theorem~\ref{thm:PSD_univKissing2}. The extension of Schoenberg's theorem by \citet{BV} is used to obtain the strongest existing semidefinite programming  upper bounds on the kissing number \cite{BV,BestBound,MV}. The kissing number problem is a particular instance of the more general spherical codes problem. Schoenberg's theorem has been used to obtain bounds on spherical codes \cite{Pfender,BVsphCodes}, as well as bounds for other problems from coding theory and discrete geometry, such as binary codes~\cite{SchrijverTervAlg}, sphere packings~\citep{KabLev78,Packing,Packing2}, distance avoiding sets~\cite{DistAvoid}, measurable chromatic number \cite{LbChr}, one-sided kissing number~\cite{OneSidedKn}.  The extension of Schoenberg's theorem by \citet{MultPsd} has been used to obtain bounds for the maximum number of equiangular lines in $\R^n$ \cite{kPointBounds}.

Schoenberg's theorem has been generalized in several ways. First, let $r\ge 0$, and pick $r$ distinct points in $(S^{n-1})$. Consider p.d. kernels invariant under the automorphisms of the sphere fixing those points, that is the stabilizer of those points in $O_n$. Schoenberg's theorem describes the case $r=0$, when no points are fixed. Next, \citet{BV} characterized the case when $K$ is a polynomial and $r=1$ point is fixed. Finally, \citet{MultPsd} characterized the case $r\le n-2$.
 In this paper we extend this idea even further in Theorem~\ref{thm:PSD_multivKissing2}. Namely, we consider  the class of continuous functions $F(x,y,Z)$ on $(S^{n-1})^{r+2}$ such that $F(\cdot,\cdot ,Z)$ is a p.d. kernel for every $Z \in (S^{n-1})^r$. There is a close connection between our result and \cite{MultPsd}, as for any fixed $Z \in (S^{n-1})^{r}$ we have that $F(\cdot,\cdot ,Z)$ is a p.d. kernel invariant under the stabilizer of $Z$. Thus
Musin's result characterizes $F(\cdot,\cdot ,Z)$ for each fixed $Z$. However, it does not fully characterize $F$ as the dependence on $Z$ is not explicit in  \cite{MultPsd} since $Z$ is assumed to be constant.

The approach in this paper differs from the approach by \citet{MultPsd}. \citet{MultPsd} uses modified Gegenbauer polynomials and the corresponding modification of the classical addition theorem for Gegenbauer polynomials~\citep{Addition}.  On the contrary, we reduce the class of considered   functions to the case where Schoenberg's theorem~ applies. To prove our results, we generalize the notion of a p.d. kernel from a kernel on a set to a kernel on a fiber bundle. Our results describe invariant under the action of $O_n$ p.d. kernels on fiber bundles generated using the unit sphere. This result can be also viewed as a continuation of another known extension of Schoenberg's theorem due to \citet{FinPSD} who generalized the theorem for group invariant p.d. kernels on compact topological spaces.


The outline of the paper is as follows. In Section~\ref{sec:prelimKissing2} we present the basic notation and concepts used throughout the paper and motivate our study.  Section~\ref{sec:kernBund} introduces fiber bundles and kernels on them and presents our main Theorems~\ref{thm:leftFixed} and~\ref{thm:PSD_multivKissing2}. Section~\ref{sec:proofsSch} contains major proofs. Further observations and ideas about future research are considered in Section~\eqref{sec:conclusionKissing2}.

\section{Preliminaries and motivation} \label{sec:prelimKissing2}

We denote the sets of real and nonnegative real numbers by $\R$ and $\R_{+}$, respectively.
We denote the unit sphere in $\R^n$  by $S^{n-1}$ and equip it with the standard measure $\omega$. Further, $O_n$ is the orthogonal group in dimension $n$, and $\Sym^n$ is  the space of $n \times n$ symmetric matrices over $\R$.

A matrix $M \in \Sym^n$ is called positive semidefinite  if $ x\tr Mx \ge 0$ for all $x \in \mathbb{R}^n$.
We use the notation $M \succeq 0$ if $M$ is positive semidefinite. We are interested in the infinite dimensional version of positive (semi-) definiteness. 
For any $V\subseteq \R^n$, let $C(V)$ be the set of real-valued continuous functions on $V$. We call kernels on $V$ the set of symmetric real continuous functions on $V \times V$:
\[\cK(V)=\{F \in C(V \times V) :  F(x,y)=F(y,x), \ \forall \ x,y \in V\}.\]
Notice that for any finite set $U$ of size $n$, the space $\cK(U)$ of kernels on $U$ is isomorphic to $\Sym^n$. A kernel $K$ on $V$ is  positive  definite (p.d.) if for any finite $U \subset V$ the restriction of $K$ to $U\times U$ is positive semidefinite.  

\begin{proposition} [ Lemma 1 in Bochner~\cite{FinPSD}]\label{prop:Bochner} Let $V\subset \R$ be a compact set equipped with a finite measure $\mu$ strictly positive on open subsets. Then $K$ is a  p.d. kernel on $V$ if and only if
for any $g(x) \in C(V)$,
\[
 \int_{V}\int_{V} K(x,y)g(x)g(y)d\mu(x)d\mu(y) \ge 0.
\]
\end{proposition}
\noindent Throughout the paper we use the following properties of p.d. kernels due to Schoenberg~\cite{Schoenberg2}. 
\begin{proposition} \label{prop:psd} Let $V\subset \R^n$ be compact, then
\begin{enumerate}[(a)]
\item A sum of finitely many p.d. kernels on $V$ is a p.d. kernel.
\item A product of finitely many p.d. kernels on $V$ is a p.d. kernel.
\item A continuous limit of a sequence of p.d. kernels on $V$ is a p.d. kernel.
\end{enumerate}
\end{proposition}


\noindent Schoenberg \cite{Schoenberg} characterized p.d. kernels on the unit sphere invariant under the action of $O_n$  in terms of Gegenbauer, or ultraspherical, polynomials.  Gegenbauer polynomials $P_d^{\alpha}(t): \ [-1,1] \rightarrow \mathbb{R}$ of order $\alpha$ and degree $d$ are inductively defined
for any $\alpha \in \R$ and $d \ge 0$, as $ P_0^\alpha(t)=1, P_1^\alpha(t)=2\alpha t$ and for $d>1$
\begin{align}
dP_d^\alpha(t)=2t(d+\alpha-1)P_{d-1}^\alpha(t)-(d+2\alpha-2)P_{d-2}^\alpha(t) \label{Gegenbauer}
\end{align}
For each fixed order $\alpha$, Gegenbauer polynomials form an orthogonal basis, with respect to the weight function $\nu=(1-t^2)^{\alpha-\frac{1}{2}}$, for univariate polynomials on the interval $[-1,1]$. Thus, every univariate polynomial $g(t)$ of degree $d$ can be represented via its Gegenbauer polynomial expansion
\begin{align}
g(t)=\sum_{i=0}^d c_i P^\alpha_i  (t), \ c_i \in \R. \label{J_exp}
\end{align}
We denote  the standard $L^2$ norm of $P^\alpha_i$ by $p_{\alpha,i}:=\int_{-1}^1\big  |P^{\tfrac{n}{2}{-}1}_k (t)\big |^2 (1-t^2)^{\alpha-\frac{1}{2}} dt$.

\begin{proposition}[Schoenberg   \cite{Schoenberg}] \label{thm:PSD_univKissing2}
Let $n\ge 2$. The kernel $K \in \cK(S^{n-1})$ is invariant under the action of $O^n$ and p.d. if and only if
there exists $c_i \ge 0$ for $i = 0,1,\dots$ such that
\begin{align}
K(x,y)=\sum_{i  \in \N} c_i P^{\tfrac{n}{2}{-}1}_i  (x \tr y ),  \label{eq:expJacKissing2}
\end{align}
where the series converges absolutely uniformly. Also, the coefficients of the expansion~\ref{eq:expJacKissing2} are given by
\begin{equation}\label{eq:ckSch}
c_k =\tfrac{1}{p_{n/2{-}1,k}}\int_{x,y \in S^{n-1}} K\lrp{x,y}P^{\tfrac{n}{2}{-}1}_k( x\tr y) d\omega(x)d\omega(y).
\end{equation}
\end{proposition}
\medskip
Given $Z=[z_1,\dots,z_r] \in (S^{n-1})^r$, let
\[\Stab_{O_n}(Z)=\{M\in O_n: MZ=Z\}\]
be the stabilizer of $Z$ in $O_n$. \citet{MultPsd} extended Schoenberg's theorem by characterizing p.d. kernels on the unit sphere invariant under the action of $\Stab_{O_n}(Z)$. To present this extension theorem, denote by $\cR(Z)$ the range of $Z$, let $\Pi_Z = Z (Z\tr Z)^{-1} Z^T$ be the orthogonal projection onto $\cR(Z)$, and let $\Pi^{^\perp}_Z  =I - \Pi_Z$, where $I$ is the identity matrix, be the orthogonal projection onto $\cR(Z)^\perp$.
\begin{proposition}[Musin \cite{MultPsd}] \label{thm:Musin}
Let $n\ge 2$, and let $Z \in (S^{n-2})^r$ be of rank $r$. The kernel $K \in \cK(S^{n-1})$ is invariant under the action of $\Stab_{O_n}(Z)$ and p.d. if and only if
there exist p.d. kernels  $c_i$ on $\{x\in \R^r: z=Z\tr y, \ y\in S^{n-1}\}$ for $i = 0,1,\dots$ such that
\begin{align}
K(x,y)=\sum_{i \in \N} c_i(Z \tr x, Z \tr y)P^{\tfrac{n-r}{2}{-}1}_i \hspace{-0.1cm}\lrp{\tfrac{\lrp{\Pi^{^\perp}_Z x}\tr\Pi^{^\perp}_Z y}{\|\Pi^{^\perp}_Z x\|\|\Pi^{^\perp}_Z y\|}}.  \label{eq:expMus}
\end{align}
\end{proposition}

When $r=1$ and $K$ is a polynomial, the result in 
Proposition~\ref{thm:Musin}  follows from the decomposition by~\citet{BV}, who use classical results on spherical harmonics, see, e.g.,~\citep[][Chapter 9]{Andrews}. 

Propositions~\ref{thm:PSD_univKissing2} and~\ref{thm:Musin} were used to obtain new upper bounds on the spherical codes problem. In this problem, the number $A(n,\theta)$  of points on $S^{n-1}$  is maximized, for which the pairwise angular distance is not smaller than some value $\theta$.
Schoenberg's theorem (Proposition~\ref{thm:PSD_univKissing2}) leads to the  linear programming upper bound for the spherical codes problem by~\citet{DGS}, and Musin's theorem (Proposition~\ref{thm:Musin}) when $r=1$ leads to the semi-definite programming bounds by~\citet{BV}. Our findings in this paper are motivated by similar upper bounds on the spherical codes problem by~\citet{Kissing1} based on the copositive reformulation of the problem by~\citet{alphaInf}. The bound in~\cite{Kissing1} exploits the so called 2-p.d. functions;
For any $r \ge 0$ and  $V\subseteq \R^n$, a function $F \in  C(V^{r+2})$ is called 2-p.d. on $V$ if for all $Z\in V^r$, $F(\cdot,\cdot,Z)$ is a  p.d. kernel on $V$.

Let  $V=S^{n-1}$, $r\ge 1$, and let $F$ be a 2-p.d. function on the unit sphere invariant under  the action of $O_n$. Implementing the upper bound in~\cite{Kissing1} requires a  characterization of $F$. It is clear that for every fixed $Z\in (S^{n-1})^r$, $F(\cdot,\cdot,Z)$ is invariant under $\Stab_{O_n}(Z)$, and thus has the form as in Proposition~\ref{thm:Musin}. However, in the context of~\cite{Kissing1} $Z$ is variable.  The question is then how to modify Proposition~\ref{thm:Musin} to make explicit the dependance on $Z$.

Although this fact is not stated explicitly, in Proposition~\ref{thm:Musin} functions $c_i$ can differ for different choices of $Z$. More precisely, for each orbit $Z^{O_n}=\{MZ: M\in O_n\}$ of $Z\in S^{n-1}$ we have a different  $c_i^{Z^{O_n}}$, but since $Z$ is fixed, this dependence of $c_i$ on the orbit of $Z$ is implicit in Proposition~\ref{thm:Musin}.  We generalize Proposition~\ref{thm:Musin} taking this dependence into account in order to characterize 2-p.d. functions. In order to simplify our presentation, we consider a generalization of p.d. kernels on a set to p.d. kernels on a fiber bundle (see next Section~\ref{sec:kernBund} for precise definitions). We characterize  p.d. kernels on fiber bundles generated by products of the unit sphere and compact sets. We use this result to characterize the set of 2-p.d. functions on the unit sphere invariant under  the action of $O_n$.

\section{Kernels on bundles} \label{sec:kernBund}
We think on kernels parameterized by a set of parameters $B$. Here not only the kernel depends on the parameters from $B$, but also the domain of the kernel might depend on the choice of the parameters from $B$.

\subsection{Bundles and their properties} \label{subsec:BunDef}
A bundle is a map $f: A \to B$,  where
$A$ is called the total space and $B$ is called the base space~\cite{ Husemoller}. For each $b \in B$, $A_b := f^{-1}(b) \subset A$ is called the fiber over $b$. We think of $A_b$ as representing a set ``parameterized by" $b$.
Our definition of a bundle is quite unrestrictive. In particular, we do not ask the fibers to be homeomorphic.  If for every  $b\in B$ the set $A_b$ is compact, we say that the bundle $f:A\to B$ is compact. In this paper we restrict ourselves to the case of compact bundles, where $A\subset \R^n$,  $B\subset \R^m$.

\begin{definition}  
As examples, we define the following bundles which we frequently use in the sequel.
\begin{enumerate}
\item Given $A\subset \R^n$ and  $B\subset \R^m$, let $\pi_{A,B}: A\times B \to B$ be the projection bundle defined by $\pi_{A,B}(a,b): = b$.
\item Given bundles $f_1:A_1 \to B$ and $f_2:A_2 \to B$,   define \[A_1\otimes A_2: = \{(a_1,a_2) \in A_1\times A_2:f_1(a_1)=f_2(a_2)\}.\]
We define the fiber product bundle $f_1 \otimes f_2: A_1\otimes A_2 \to B$ as \[f_1\otimes f_2(a_1,a_2): = f_1(a_1).\]
The fiber product is also called the Whitney sum. It has the property  that for every $b\in B$, $(A_1 \otimes A_2)_b = (A_1)_b \times (A_2)_b$.
\item Given a bundle $f:A \to B$ and  $U\in \R^k$, let
$\cyl{U}{f}: U\times A \to B$ be the bundle such that $_{U \times }f(u,a): = f(a)$. In the case $U = S^{n-1}$, we call $\cyl{S^{n-1}}{f}$ a cylinder.
\end{enumerate}
\end{definition}

Now, we introduce the notion of kernel on a bundle. The idea is that for each $b \in B$ we have a kernel on $A_b$, and the dependence on $b$ is continuous. Given a bundle $f: A\to B$, we define a kernel on $f$ to be a continuous map $K: A\otimes A \to \R$ such that $K_b:A_b \times A_b \to \R$ is a kernel for each $b\in B$.  We say $K$ is p.d. on $f$ if $K_b$ is p.d. for each $b \in B$. 

We say that a bundle $f: U\to B$ is a subbundle of $g: A \to B$  if $U\subseteq A$, and $f=g|_{U}$. We call $f:U\to B$ a projection subbundle if it is a subbundle of some projection bundle. 

\begin{remark}\label{rem:2PSD-bundles}
Given any projection bundle $\pi_{A,B}: A\times B \to B$, we have $\pi_{A,B}\otimes \pi_{A,B} \cong \pi_{A\times A,B}$, and thus every p.d. kernel on a subbundle $f:U\to B$ of $\pi_{A,B}$ is in correspondence with a continuous map $K:U\times U \times B \to \R$ such that for each $b\in B$, $K(\cdot,\cdot,b)=K_b$ is a p.d. kernel on $U_b$. In the sequel we abuse the notation and make no difference between a kernel on projection subbundle $f:U\to B$ and its corresponding continuous map. 
\end{remark}

Our last definition is the action of a group on a bundle. Given bundle $f:A \to B$ and group $G$, for $G$ to act on $f$ means that $G$ acts both on $A$ and on $B$, and both actions are consistent with $f$. That is,  for all $g \in G$ and $a\in A$, $f(a)^g = f(a^g)$. We denote the orbit of $a\in A$ under $G$  by $a^G:=\{a^g:g \in G\}$, and let $\orb_G(A) = \big \{a^G: a \in A \big \}$  be the set of orbits of $A$. We  define $b^G$ and $\orb_G(B)$ analogously.

When $G$ acts on a bundle $f:A\to B$, it is natural to define  the $G$-orbit bundle of $f$  as $\orb_G(f):\orb_G(A) \to \orb_G(B)$ such that $\orb_G(f)(a^G)=f(a)^G$. Notice that in the $G$-orbit bundle, for any $b^G \in \orb_G(B)$ we have $[\orb_G(A)]_{b^G} = \orb_G(A_b)$.

Now, we propose an extension of the group acting on $B$ to a group acting on projection bundle $\pi_{A,B}$.  Assume $G$ acts on $B$. We define the vertical action of $G$ on $\pi_{A,B}$ by fixing the elements of $A$;
that is, for all $a\in A$, $b\in B$ and $g \in G$ we define the action of $G$ on $A\times B$ by $(a,b)^g := (a,b^g)$. Notice that this action and the action of $G$ on $B$ are consistent with $\pi_{A,B}$, and thus define an action on $\pi_{A,B}$.
Moreover, $\orb_G(\pi_{A,B}) = \pi_{A,\orb_G(B)}$. For any projection subbundle $f: U\to B$, we say that $G$ acts vertically on $f$ if the action of $G$ is the restriction of the vertical action on the corresponding projection bundle.  Notice that this is the case only if for any $b\in B$ and $g\in G$ we have $U_b=U_{b^g}$.

In general, looking at $\orb_G(f)$ is not enough to characterize p.d. kernels on $f$ invariant under the action of $G$ as kernels are bivariate functions and thus one should look at 2-orbits, instead of 1-orbits. One exception is the case of vertical actions, as the following straightforward proposition shows.
\begin{proposition}\label{prop:vertical}
Let $A\subset \R^n$ and  $B\subset \R^m$. Assume that $G$ is a topological group that acts on $B$, and endows $\orb_G(B)$ with the usual topology. Let $f: A\to B$ be a projection subbundle such that $G$  acts vertically on $f$.
Let $K$ be a kernel on $f$ invariant under the vertical action of $G$. Define the function $K^{G}$ as
\[K^{G}(a_1, a_2,b^G):=  K(a_1,a_2,b) \text{ for all } b\in B, \ a_1,a_2\in A_b.\]
Then  $K^{G}$ is a kernel on $\orb_G(f)$, and $K$ is p.d. if and only if $K^{G}$ is p.d.
\end{proposition}
\begin{proof}
The result follows from the definition of the vertical action and the $G$-orbit bundle.
\end{proof}

This work is motivated by 2-p.d. functions introduced in \cite{Kissing1}.  Remark \ref{rem:2PSD-bundles} explains that 2-p.d. functions on $A$ are p.d. kernels on $\pi_{A,A^r}$ for some given $r$.
 The interest in \cite{Kissing1} is in $(r+2)$-variate 2-p.d. functions on $S^{n-1}$ invariant under the natural action of $O_n$ on $(S^{n-1})^{r+2}$. In the language of kernels on fiber bundles, those are p.d. kernels on $\pi_{ S^{n-1},(S^{n-1})^r}$ invariant under the natural action of $O_n$ on $\pi_{ S^{n-1},(S^{n-1})^r}$. In the rest of this paper we
 characterize such kernels.

\subsection{Main results}

Next we present two generalizations of Schoenberg's theorem (Propoition~\ref{thm:PSD_univKissing2}).
Given a bundle $f:A \to B$, we define the horizontal action of $O_n$ on the cylinder $\cyl{S^{n-1}}{f}$ by $(x,a)^M = (Mx,a)$ and $b^M =b$,
for each $x \in S^{n-1}$,  $a \in A$, $b\in B$  and $M \in O_n$. In our first Theorem we characterize the p.d. kernels on cylinders, invariant under the horizontal action of $O_n$ .

\begin{theorem} \label{thm:leftFixed} Let $n\ge 2$, $r>0, m>0$, and let $A\in \R^r, B\in \R^m$. Let $f:A \to B$ be a compact bundle. Then a kernel $K$ on the cylinder $\cyl{S^{n-1}}{f}$ is p.d. and invariant under the horizontal  action of $O_n$ if and only if there are p.d. kernels $c_i$ on $f$, $i=0,1,\dots$ such that for all $u_1,u_2 \in S^{n-1}$, all $b \in B$ and $a_1,a_2 \in A_b$
\begin{align}
K_b\lrp{\smat{a_1\\u_1},\smat{a_2\\u_2}}=\sum_{i \in \N} (c_i)_b(a_1,a_2)P^{\tfrac{n}{2}{-}1}_i  (u_1 \tr u_2 ). \label{eq:expBasic}
\end{align}
Also, the coefficients of the expansion are given by
\begin{equation}\label{eq:ck}
(c_k)_b(a_1,a_2) = \tfrac{1}{p_{n/2{-}1,k}}\int_{u_1,u_2 \in S^{n-1}} K_b\lrp{\smat{a_1\\u_1},\smat{a_2\\u_2}}P^{\tfrac{n}{2}{-}1}_k( u_1\tr u_2) d\omega(u_1)d\omega(u_2).
\end{equation}
\end{theorem}

Notice that $\pi_{S^{n-1}, (S^{n-1})^r}$ is isomorphic to the cylinder $\cyl{ S^{n-1}}{\id{(S^{n-1})^r}}$, where $\id{(S^{n-1})^r}$ is the identity bundle on $ (S^{n-1})^r$.
Consider the (natural) action of $O_n$ on $\pi_{S^{n-1}, (S^{n-1})^r}$. This action is not  horizontal, and thus Theorem~\ref{thm:leftFixed} does not apply.
Our second theorem describes  p.d. kernels on the bundle $\pi_{ S^{n-1},(S^{n-1})^r}$ which are invariant under the  action of $O_n$. Given $r>0$,  define
\begin{align*}
 \Lambda^{r} & =\lrc{Y \in\Sym^{r}: Y\succeq 0, \ Y_{ii}=1 \text{ for all } i\in \{1,\dots,r\}}.
\end{align*}
\begin{theorem} \label{thm:PSD_multivKissing2}
Let $r\ge 0$ and $n \ge  r+2$.
A kernel $K$ on $\pi_{ S^{n-1},(S^{n-1})^r}$ is invariant under the action of $O_n$ and p.d. if and only if there are p.d. kernels $c_i$, $i=0,1,\dots$  on the projection subbundle  $f: \lrc{\smat{1& y\tr \\ y & Y} \in \Lambda^{r+1}: Y \succ 0} \to \lrc{Y\in \Lambda^{r}: Y\succ 0} $, such that for all $x,y \in S^{n-1}$ and $Z \in (S^{n-1})^r$,
\begin{align}
 K_Z(x,y)=\sum_{i \in \N} (c_i)_{Z\tr Z}(Z \tr x, Z \tr y)P^{\tfrac{n-r}{2}{-}1}_i \hspace{-0.1cm}\lrp{\tfrac{\lrp{\Pi^{^\perp}_Z x}\tr\Pi^{^\perp}_Z y}{\|\Pi^{^\perp}_Z x\|\|\Pi^{^\perp}_Z y\|}}. \label{eq:expJacGenKissing2}
\end{align}
\end{theorem}

The expressions in Theorem \ref{thm:PSD_multivKissing2} and Proposition \ref{thm:Musin} are very similar. Indeed, the difference is that Proposition \ref{thm:Musin} does not present $c_i$ as a function of $Z \tr Z$ since $Z$ is assumed to be fixed. Notice that Proposition \ref{thm:Musin} follows from Theorem \ref{thm:PSD_multivKissing2}, as given any $Z \in (S^{n-2})^r$ and any  p.d. kernel $K$ on $S^{n-1}$ invariant under the action of $\Stab_{O_n}(Z)$, there is
p.d. kernel $\hat K$  on $\pi_{ S^{n-1},(S^{n-1})^r}$ invariant under the action of $O_n$ such that $\hat{K}_Z(x,y) = K(x,y)$. 

\section{Proofs of main theorems} \label{sec:proofsSch}
In this section we present the proofs of Theorems \ref{thm:leftFixed} and \ref{thm:PSD_multivKissing2}.

\subsection{Proof of Theorem \ref{thm:leftFixed}} \label{subsec:proofThm1}
Let  $K$ be a kernel on $\cyl{S^{n-1}}{f}$. The ``only if" part of the statement follows from Propositions~\ref{prop:psd} and~\ref{thm:PSD_univKissing2}. To prove the converse, let  $K$ be p.d. and  invariant under the horizontal action of $O_n$.
Let $b\in B$ and $a_1,a_2\in A_b$. The kernel
\[
G_{b}^{a_1,a_2}(u_1,u_2)= K_b\lrp{\smat{a_1\\u_1},\smat{a_2\\u_2}}+K_b\lrp{\smat{a_2\\u_1},\smat{a_1\\u_2}} + K_b\lrp{\smat{a_1\\u_1},\smat{a_1\\u_2}}+K_b\lrp{\smat{a_2\\u_1},\smat{a_2\\u_2}}
\]
is p.d. on $S^{n-1}$ and invariant under $O_{n}$.
From Schoenberg's theorem (Proposition~\ref{thm:PSD_univKissing2}) we have
\begin{align}
G_{b}^{a_1,a_2}(u_1,u_2)= \sum_{k \ge 0} (d_i)_b(a_1,a_2) P^{\tfrac{n}{2}{-}1}_k  (u_1 \tr u_2 ), \label{eq:expD}
\end{align}
where the $(d_i)_b(a_1,a_2)$ are nonnegative, and the series \eqref{eq:expD} converges absolutely uniformly.

As $O_n$ acts transitively on $S^{n-1}$, we have
\[K_b\lrp{\smat{a_1\\u_1},\smat{a_2\\u_2}} =  K_b\lrp{\smat{a_1\\u_2},\smat{a_2\\u_1}} =  K_b\lrp{\smat{a_2\\u_1},\smat{a_1\\u_2}},\]
thus,
\begin{align*}
  G_{b}^{a_1,a_2}(u_1,u_2) & =  2K_b\lrp{\smat{a_1\\u_1},\smat{a_2\\u_2}}+ K_b\lrp{\smat{a_1\\u_1},\smat{a_1\\u_2}}+K_b\lrp{\smat{a_2\\u_1},\smat{a_2\\u_2}}\\
   G_b^{a_i,a_i}(u_1,u_2) & =  4K_b\lrp{\smat{a_i\\u_1},\smat{a_i\\u_2}} & (i=1,2).
\end{align*}
Defining $(c_k)_b(a_1,a_2) =  (d_k)_b(a_1,a_2)-\tfrac{1}{4} (d_k)_b(a_1,a_1)-\tfrac{1}{4} (d_k)_b(a_2,a_2)$ and using \eqref{eq:expD}, we obtain
\begin{align}
K_b\lrp{\smat{a_1\\u_1},\smat{a_2\\u_2}}= & \tfrac{1}{2}\lrp{G_{b}^{a_1,a_2}(u_1,u_2)-\tfrac{1}{4}G_b^{a_1,a_1}(u_1,u_2)-\tfrac{1}{4}G_b^{a_2,a_2}(u_1,u_2)} \nonumber \\
= &  \ \tfrac{1}{2} \sum_{k \in \N} (c_k)_b(a_1,a_2) P^{\tfrac{n}{2}{-}1}_k  (u_1 \tr u_2 ). \label{eq:expC2}
\end{align}
\begin{remark} \label{rem:0}
Notice that Schoenberg's theorem can not be applied directly to $K_b\lrp{\smat{a_1\\u_1},\smat{a_2\\u_2}}$, as this is not necessarily a p.d. kernel for all $b \in B$ and  $a_1, a_2 \in A_b$. Intuitively the reason for this is that this function does not correspond to a ``principal submatrix" of $K_b$ when $a_1 \neq a_2$.
\end{remark}

Next, we argue that $c_k$'s are p.d. kernels on $f$. Fix $k \ge 0$, then we claim that
\begin{claim} \label{cl:intExp}
For every $b\in B$, $a_1,a_2\in A_b$,
\begin{equation*}
(c_k)_b(a_1,a_2) = \tfrac{1}{p_{n/2{-}1,k}}\int_{u_1,u_2 \in S^{n-1}} K_b\lrp{\smat{a_1\\u_1},\smat{a_2\\u_2}}P^{\tfrac{n}{2}{-}1}_k( u_1\tr u_2) d\omega(u_1)d\omega(u_2) \nonumber
\end{equation*}
\end{claim}
 Fix $b \in B$. Claim~\ref{cl:intExp} and the continuity of $K$ imply that $(c_k)_b$ is continuous.  From \eqref{eq:ck} it follows that $(c_k)_b(a_1,a_2)=(c_k)_b(a_2,a_1)$ for all $a_1,a_2\in A_b$. Hence $(c_k)_b$ is a kernel.

 From our assumptions, $A_b$ is compact. We use Proposition~\ref{prop:Bochner} to show that $(c_k)_b$ is p.d. on $S^{n-1}\times A_b$.  Let $h\in C(A_b)$ be given, then
\begin{align*}
&\int_{a_1,a_2\in A_b}(c_k)_b(a_1,a_2)h(a_1)h(a_2) d\mu(a_1)d\mu(a_2) \\
&\qquad=\int_{\substack{ a_1,a_2\in A_b\\
u_1,u_2 \in S^{n-1}
}} K_b\lrp{\smat{a_1\\u_1},\smat{a_2\\u_2}}P^{\tfrac{n}{2}{-}1}_k( u_1\tr u_2)  h(a_1)h(a_2) d\omega(u_1)\dots d\mu(a_2)\\
&\qquad \ge 0.
\end{align*}
The inequality holds since $K_b$ and $P^{\tfrac{n}{2}{-}1}_k$ are p.d. on $S^{n-1}\times A_b$, and from Proposition \ref{prop:psd}.b their product is p.d. too.

To finish, we prove Claim \ref{cl:intExp}. We know that  $\big\{P^{\tfrac{n}{2}{-}1}_i\big \}_{i\in \N}$ is a sequence of orthogonal polynomials on $S^{n-1}$ under $\omega$. Fix $k\ge 0$, $b\in B$ and $a_1,a_2\in A_b$. From Schoenberg's theorem~\ref{thm:PSD_univKissing2} and the definition of $(c_i)_b(a_1,a_2)$ in~\ref{eq:expC2}, we have that
 $\sum_{i \ge 0} (c_i)_b(a_1,a_2)P^{\tfrac{n}{2}{-}1}_i( u_1\tr u_2)$ converges absolutely uniformly on $S^{n-1}\times S^{n-1}$. Hence, as $P_k$ is continuous on $[-1,1]$ and therefore bounded, the series
 \[\sum_{i \ge 0}(c_i)_b(a_1,a_2)P^{\tfrac{n}{2}{-}1}_i( u_1\tr u_2)P^{\tfrac{n}{2}{-}1}_k( u_1\tr u_2)\]
 converges absolutely uniformly too. Therefore
\begin{align*}
\hspace{-0,7cm} \int_{u_1,u_2 \in S^{n-1}}& K_b \big (\smat{a_1\\u_1},  \smat{a_2\\u_2},  b \big )P^{\tfrac{n}{2}{-}1}_k( u_1\tr u_2) d\omega(u_1)d\omega(u_2) \\
&= \int_{u_1,u_2 \in S^{n-1}} \sum_{i \ge 0} (c_i)_b(a_1,a_2)P^{\tfrac{n}{2}{-}1}_i( u_1\tr u_2)P^{\tfrac{n}{2}{-}1}_k( u_1\tr u_2)d\omega(u_1)d\omega(u_2)\nonumber \\
&= \sum_{i \ge 0}  (c_i)_b(a_1,a_2)\int_{u_1,u_2 \in S^{n-1}} P^{\tfrac{n}{2}{-}1}_i( u_1\tr u_2)P^{\tfrac{n}{2}{-}1}_k(  u_1\tr u_2)d\omega(u_1)d\omega(u_2)\nonumber \\
&= \ p_{n/2{-}1,k}(c_k)_b(a_1,a_2). \nonumber
\end{align*}

\subsection{Proof of Theorem~\ref{thm:PSD_multivKissing2}} \label{subsec:proofThm2}

The idea of the proof is to relate kernels on $\pi_{ S^{n-1},(S^{n-1})^r}$ to kernels on cylinders over $S^{n-r-1}$ and apply Theorem~\ref{thm:leftFixed}. We do this via continuous transformations using Lemmas~\ref{lem:ContinMaps} and \ref{lem:T1T2} below.

\begin{lemma}\label{lem:ContinMaps} Let $f:C \to B$ be a compact bundle. Let $T:A\rightarrow C$ be a continuous function. Let $K$ be a kernel on $f$. Define  $L$ by $L_b(a_1,a_2)=K_b(T(a_1),T(a_2))$ for all $b\in B$, $a_1,a_2\in A_b$.
\begin{enumerate}
\item   $L$ is a kernel on $f \circ T$.
\item  If $K$ is p.d. on $f$, then $L$ is p.d. on $f \circ T$.
\item If $T$ is surjective and $L$ is p.d. on $f \circ T$, then  $K$ is p.d. on $f$.
\end{enumerate}
\end{lemma}
\begin{proof}
As $T$ and $K$ are continuous, $L$ is continuous by definition. Also, for any $b\in B$ and any $a_1,a_2 \in A_b$, we have $L_b(a_1,a_2) = K_b(T(a_1),T(a_2)) = K_b(T(a_2),T(a_1)) =  L_b(a_2,a_1)$. Thus $L$ is a kernel on $f \circ T$.
Now assume $K$ is p.d. on $f$. For any $k >0$, any $b\in B$, and any $a_1,a_2,\dots,a_k \in A_b$, we have that  $\smat{L_b(a_i,a_j)}_{1 \le i,j \le k} = \smat{K_b(T(a_i),T(a_j))}_{1 \le i,j \le k} \succeq 0$. Thus $L$ is p.d. on $f \circ T$.
Now assume $T$ is surjective and $L$ is p.d. on $f \circ T$. Then given  $k >0$, any $b\in B$, and any $c_1,c_2,\dots,c_k \in C_b$
there are $a_1,a_2,\dots,a_k \in A_b$  such that $c_i = T(a_i)$. Thus  the matrix $\smat{K_b(c_i,c_j)}_{1 \le i,j \le k} =\smat{ L_b(a_i,a_j)}_{1 \le i,j \le k} $ is p.d.. Therefore $K$ is p.d. on $f$.
\end{proof}

Given $Z \in (S^{n-1})^r$. For any $x \in \R^n$ we can write
$x=\Pi_Z x+\Pi^{^\perp}_Z x,$
 where  $\Pi_Z x \in \cR(Z)$ and $\Pi^{^\perp}_Z x \in \cR(Z)^\perp$. To prove Theorem~\ref{thm:PSD_multivKissing2}, we exploit the fact that  $\Stab_{O_n}(Z)$ fixes  $\Pi_Z S^{n-1}$ and acts transitively on $\Pi^{^\perp}_{Z}S^{n-1}$.  Therefore in the following steps, for every $x\in S^{n-1}$ we separate the fixed component $\Pi_{Z}x$ and exploit the symmetry of the varying component  $\Pi^{^\perp}_{Z}x$.

We have $\dim(\cR(Z)^\perp) = n - \dim(\cR(Z)) = n- \rank Z$, and therefore $\cR(Z)^\perp$ is isomorphic to $\R^{n- \rank Z}$. Namely, there is an isomorphism $\phi_Z$ between $R^{n-\rank Z -1}$ and $\cR(Z)^\perp$. Analogously,  $\cR(Z)$ has dimension $\rank Z$ and thus is isomorphic to $\R^{\rank Z}$, and there is an isomorphism $\gamma_Z$ between $R^{\rank Z}$ and $\cR(Z)$. Let
\[\cS = \{ Z \in (S^{n-1})^r: \rank Z  = r\}.\]
The set $\cS$ is dense in $(S^{n-1})^r$. In the sequel we restrict most of our arguments to $\cS$, to avoid ``singularities" in further proofs and definitions. In particular, the dependence on $Z$ of the isomorphisms $\phi_Z$ and $\gamma_Z$ can not be continuous in the whole $(S^{n-1})^r$, but when we restrict ourselves to $\cS$,  $\phi_Z$ and $\gamma_Z$ can be chosen continuous. For instance, let $Ort: \R^{n\times r} \to \R^{n\times (n-r)}$ be such that $Ort(Z)$ provides an orthonormal basis of $\cR(Z)^\perp$ for any $Z\in \R^{n\times r}$ of rank $r$. Then the isomorphism between $\cR(Z)^\perp$ and $\R^{n- r}$ can be viewed as the bijection:
\begin{align*}
\phi_Z:  \R^{n-r}\to \cR(Z)^\perp, \ \phi_Z(v)=Ort(Z) v \text{ for all } v\in \R^{n-r}.
\end{align*}
We can construct a continuous isomorphism between $R^{r}$ and $\cR(Z)$  as the following bijection
\begin{align*}
\gamma_Z: \R^{r} \to \cR(Z), \ \gamma_Z(u)=Z(Z\tr Z)^{-1} u \text{ for all } u\in \R^{r}.
\end{align*}
We are particularly interested in the isomorphism between $\Pi_Z S^{n-1}\subset \cR(Z)$  and
 \begin{align*}
B_Z: =\gamma_Z^{-1} \circ \Pi_Z S^{n-1}= \{Z\tr x:  x\in S^{n-1}\}.
\end{align*}

Notice that for any $Z\in \cS$ we can send  $x\in \Pi^{^\perp}_{Z}S^{n-1}$  to the unit sphere in $\R^{n- r}$ by normalizing $x$. Then, since $\Stab_{O_n}(Z)$ is isomorphic to $O_{n-r}$, any action of $\Stab_{O_n}(Z)$ on $\Pi^{^\perp}S^{n-1}$ can be associated with an action of $O_{n-r}$ on $S^{n-r -1}$. Hence we can use the result for the orthogonal group acting on the unit sphere, described in Theorem~\ref{thm:leftFixed}. To formalize this procedure, we need the following lemma.





\begin{lemma}\label{lem:T1T2} The maps
\[\begin{array}{rcccc}
T_1: S^{n-r-1} \times \{(u,Z): Z\in \cS, \ u\in B_Z\}  \rightarrow & \hspace{-0.1cm} \{(x, Z):Z \in \cS,\, x \in S^{n-1}\}  \\[3pt]
        \smat{v \\ u \\ Z} \mapsto  & \smat{ \phi_Z(v)\sqrt{1-\|\gamma_Z(u)\|^2} + \gamma_Z(u) \\ Z}
\end{array}\]
\[\begin{array}{rcccc}
 T_2:  \{(x,Z):Z \in \cS, \ x \in S^{n-1}\setminus\cR(Z)\} \rightarrow &  S^{n-r-1}\times \{(u,Z): Z\in \cS, u \in B_Z\}  \\[5pt]
        \smat{x \\Z} \mapsto  &  \smat{  \tfrac{\phi_Z^{-1}(\Pi^{^\perp}_Z x) }{\|\phi_Z^{-1}(\Pi^{^\perp}_Z x)\|} \\ \gamma_Z^{-1}(\Pi_Z x) \\ Z}
\end{array}\]
are continuous, $T_1$ is surjective, $T_2$ is injective, and $T_1 \circ T_2 = \id_{ \{(x,Z):Z \in \cS,\, x \in S^{n-1}\setminus\cR(Z)\} }$.
\end{lemma}
\begin{proof}
Continuity follows from the continuity of $\gamma_Z$, $\phi_Z$ and their inverses. It is a straightforward calculation to check that $T_1$ is a surjection and  $T_1 \circ T_2 = \id_{ \{(x,Z):Z \in \cS,\, x \in S^{n-1}\setminus\cR(Z)\} }$.
$T_2$ is injective since for any $x_1,x_2\in \R^n$, if $x_1\neq x_2$ then either $\Pi_Z x_1\neq \Pi_Zx_2$ or $\Pi^{^\perp}_Z x_1 \neq \Pi^{^\perp}_Z x_2$.
\end{proof}
\begin{proof}[Proof of Theorem~\ref{thm:PSD_multivKissing2} ] First, denote
\[\B := \lrc{(u,Z):Z \in \cS,\, u\in B_Z}=\lrc{(Z \tr x,Z):Z \in \cS,\, x\in S^{n-1}}, \]
and consider the compact projection subbundle
$f:\B \to \cS.$  Let $K$ be a kernel  on  $\pi_{ S^{n-1},(S^{n-1})^r}$. Throughout the proof we use a function $L$ on the cylinder  $\cyl{S^{n-r-1}}{f}$ defined by
\begin{align}
L_Z\lrp{\smat{u_1 \\v_1},\smat{u_2\\v_2}}=K\lrp{T_1\lrp{\smat{u_1 \\v_1\\Z}},T_1\lrp{\smat{u_2\\v_2\\Z}}}, \label{def:L}
\end{align}
for all $v_1,v_2\in S^{n-r-1}$, $Z\in \cS$, $u_1,u_2\in  B_Z$. By our assumptions $K$ is a kernel on $\pi_{ S^{n-1},(S^{n-1})^r}$. Since $T_1$ is surjective by Lemma~\ref{lem:T1T2},  we have that $L$ is a kernel on $\cyl{S^{n-r-1}}{f}$ by Lemma~\ref{lem:ContinMaps}.

For the ``only if" direction of the theorem, let $K$ be a kernel on  $\pi_{ S^{n-1},(S^{n-1})^r}$ that has expansion \eqref{eq:expJacGenKissing2}. Then $K$  is invariant under the action of $O_n$  as it is continuous and depends only on inner products of $x,y,z_1,\dots,z_r$ on a dense subset of its domain.  To show that $K$ is p.d., notice that   expansion~\eqref{eq:expJacGenKissing2} of $K$ implies  $L$ has an expansion of type~\eqref{eq:expBasic}. Hence $L$ is a p.d. kernel on $\cyl{S^{n-r-1}}{f}$ by Theorem~\ref{thm:leftFixed}. Therefore $K$ is p.d. when restricted to a kernel on $\pi_{S^{n-1},\cS}$ by Lemma~\ref{lem:ContinMaps}, and thus $K$ is p.d. by continuity.

For the ``if" direction of the theorem, let  $K$ be  p.d. on the projection bundle $\pi_{ S^{n-1},(S^{n-1})^r}$ and invariant under the action of $O_n$. That is, $K_{MZ}(Mx,My)=K_Z(x,y)$ for all $M\in O_n, x,y\in S^{n-1}$ and $Z\in (S^{n-1})^r$. Then for any $Z\in(S^{n-1})^r$,  $K_Z$ is invariant under the action of $\Stab_{O_n}(Z)$. Using Lemma~\ref{lem:ContinMaps}, we obtain that $L$ is a p.d. kernel on $\cyl{S^{n-r-1}}{f}$. As $\Stab_{O_n}(Z)$ fixes $\Pi_Z S^{n-1}$ and acts transitively on $S^{n-1} \cap \cR(Z)^\perp$, we have that   $L$ is invariant under the horizontal action of $O_{n-r}$ on $\cyl{S^{n-r-1}}{f}$.
From Theorem~\ref{thm:leftFixed} there are p.d. kernels $d_i$ on $f$, $i=0,1,\dots$ such that for all $v_1,v_2 \in S^{n-r-1}$, all $Z \in \cS$ and $u_1,u_2 \in B_Z$
\begin{align}
L_Z\lrp{\smat{u_1 \\v_1},\smat{u_2\\v_2}}=\sum_{i \ge 0} (d_i)_Z(u_1,u_2)P^{\tfrac{n-r}{2}{-}1}_i  (v_1 \tr v_2 ). \label{eq:expL}
\end{align}
Now, $T_1 \circ T_2 = \id_{ \{(x,Z):Z \in \cS,\, x \in S^{n-1}\setminus\cR(Z)\} }$ from Lemma~\ref{lem:T1T2}. Thus for any $Z \in \cS$ and $x_1,x_2 \in S^{n-1}\setminus \cR(Z)$, we have
\begin{align}
\hspace{-0.7cm} K_Z\lrp{x_1,x_2} & = L_Z\lrp{T_2\lrp{\smat{x_1\\Z}},T_2\lrp{\smat{x_2\\Z}}} \nonumber \\
& = \sum_{i \ge 0} (d_i)_Z\lrp{ \gamma_Z^{-1}(\Pi_Z x_1), \gamma_Z^{-1}(\Pi_Z x_2)} P^{\tfrac{n-r}{2}{-}1}_i \left(\lrp{\tfrac{\phi_Z^{-1}(\Pi^{^\perp}_Z x_1) }{\|\Pi^{^\perp}_Z x_1\|}} \tr  \tfrac{\phi_Z^{-1}(\Pi^{^\perp}_Z x_2) }{\|\Pi^{^\perp}_Z x_2\|}\right) \nonumber \\
&= \sum_{i \ge 0}(d_i)_Z(Z\tr x_1,Z\tr x_2)
P^{\tfrac{n-r}{2}{-}1}_i \left( \tfrac{\lrp{\Pi^{^\perp}_Z x_1}\tr\Pi^{^\perp}_Z x_2}{\|\Pi^{^\perp}_Z x_1\|\|\Pi^{^\perp}_Z x_2\|}\right), \label{eq:Fdraft}
\end{align}
where we have used
$\phi_Z^{-1}(a_1) \tr \phi_Z^{-1}(a_2) = a_1 \tr a_2$.

To finish, we specify the form of the $d_i$'s. Let $k \ge 0$. By Lemma~\ref{lem:ContinMaps} we have that $d_k$ is a p.d. kernel on $f:\B \to \cS$ . Notice that $O_n$ acts vertically on $f$ since for every $M \in O_n$, $x \in S^{n-1}$ and $Z \in \cS$, we have $(Z\tr x,Z)^M = (Z \tr x,MZ)$.  Next we show that $d_k$ is invariant under this vertical action. That is, that for all $Z\in \cS$, $v_1,v_2\in B_Z$ and $M\in O_n$ we have
$(d_k)_{Z}(v_1,v_2) = (d_k)_{MZ}(v_1,v_2)$. By~\eqref{eq:ck} from Theorem~\ref{thm:leftFixed},
\begin{align*}
(d_k)_{MZ}(v_1,v_2) & =\tfrac{1}{p_{(n-r)/2{-}1,k}}\int_{u_1,u_2 \in S^{n-r-1}} L_{MZ}\lrp{\smat{u_1\\v_1},\smat{u_2\\v_2}}P^{\tfrac{n-r}{2}{-}1}_k( v_1\tr v_2) d\omega(v_1)d\omega(v_2).
\end{align*}
Therefore it is enough to show that $L$ is invariant under the vertical action of $O_n$ on $\cyl{S^{n-r-1}}{f}$. By construction of $\phi_Z, \gamma_Z$
 we have
\[\phi_{MZ}(v)=M \phi_{Z}(v), \ \gamma_{MZ}(u)=M \gamma_{Z}(u), \text{ and } \|\gamma_{MZ}(u)\|=\|\gamma_{Z}(u)\|.\]
Therefore
\begin{align*}
L_{MZ}\lrp{\smat{u_1\\v_1},\smat{u_2\\v_2}} & =K\lrp{T_1 \hspace{-0.1cm}\lrp{\smat{u_1 \\v_1\\ MZ}},T_1 \hspace{-0.1cm}\lrp{\smat{u_2\\v_2\\ MZ}}}\\
 &=K\lrp{\smat{M&0 \\0& M}T_1 \hspace{-0.1cm}\lrp{\smat{u_1 \\v_1\\ Z}},\smat{M&0 \\0& M} T_1 \hspace{-0.1cm}\lrp{\smat{u_2\\v_2\\ Z}}}=L_{Z}\lrp{\smat{u_1\\v_1},\smat{u_2\\v_2}},
\end{align*}
where the last but one equality holds by invariance of $K$ under $O_n$. Now, we can reduce ourselves to the orbit bundle of $f$ under the vertical action of $O_n$.
We have that $ \orb_{O_n}(\cS)=\interior \Lambda^{r}$,  $\orb_{O_n}(\B)= \lrc{\smat{1&y\tr\\y&Y}\in \Lambda^{r+1}: Y\succ 0}$ and $\orb_{O_n} (f)$ is a projection. From Proposition~\ref{prop:vertical} there is a p.d. kernel  $c_i$ on $\orb_{O_n}(f): \lrc{\smat{1&y\tr\\y&Y}\in \Lambda^{r+1}: Y\succ 0} \to \lrc{Y\in \Lambda^{r}: Y\succ 0} $ such that  for all $Z\in \cS$, $v_1,v_2\in B_Z$ we have that
$(c_k)_{Z\tr Z}(v_1,v_2)= (d_k)_Z(v_1,v_2)$.
\end{proof}

\section{Further observations} \label{sec:conclusionKissing2}

One question for further research is what shape the coefficients $c_i$ from expansion~\eqref{eq:expJacGenKissing2} could have. We show one result in this direction: a Gegenbauer polynomial of order $\lrp{\tfrac{n-r}{2}{-}1}$ used in~\eqref{eq:expJacGenKissing2} on $\pi_{S^{n-1},(S^{n-1})^{r}}$ can be considered as a p.d. kernel  on $\pi_{S^{n-1},(S^{n-1})^{r+1}}$, and therefore it can be expressed in a series of the form~\eqref{eq:expJacGenKissing2} with coefficients $c_i$ of a particular but  rather complex form. To simplify the notation, given $x,y\in S^{n-1}$ and $Z\in (S^{n-1})^r$, define
 \[\ip{x}{y}_Z := (\Pi^{^\perp}_Zx)\tr \Pi^{^\perp}_Zy = x\tr y-(Z \tr x)\tr (Z\tr Z)^{-1} Z \tr y.\]
\begin{proposition} \label{prop:addition}
Let $r\in \N$,  $x,y,q\in S^{n-1}$ and $ Z\in (S^{n-1})^r$, then
\begin{align*}
\hspace{-0.9cm}
 P^{\tfrac{n-r}{2}-1}_k\hspace{-0.15cm} \lrp{\tfrac{\ip{x}{y}_Z}{\sqrt{\ip{x}{x}_Z}\sqrt{\ip{y}{y}_Z}}}\hspace{-0.05cm}
&= \sum_{i=0}^k \hspace{-0.05cm} c_{k,i}^{n,r}
\lrp{\tfrac{\ip{x}{x}_{[Z q]}}{\ip{x}{x}_Z}}^{i/2} \hspace{-0.1cm}\lrp{\tfrac{\ip{y}{y}_{[Z q]}}{\ip{y}{y}_Z}}^{i/2}\hspace{-0.1cm} P^{\tfrac{n-r}{2}-\frac{3}{2}}_i\hspace{-0.1cm}\lrp{\hspace{-0.05cm}\tfrac{\ip{x}{y}_{[Z q]}}{\sqrt{\ip{x}{x}_{[Z q]}}\sqrt{\ip{y}{y}_{[Z q]}}}\hspace{-0.05cm}} \\
& \quad \ \cdot  P^{\tfrac{n-r}{2}+i-1}_{k-i}\lrp{\tfrac{\ip{x}{q}_Z}{\sqrt{\ip{x}{x}_Z}\sqrt{\ip{q}{q}_Z}}}
 P^{\tfrac{n-r}{2}+i-1}_{k-i}\lrp{\tfrac{\ip{y}{q}_Z}{\sqrt{\ip{y}{y}_Z}\sqrt{\ip{q}{q}_Z}}}
\end{align*}
where $c_{k,i}^{n,r}$ are positive constants.
\end{proposition}
\begin{proof} We use the addition theorem for Gegenbauer polynomials.  Let $\alpha>0$ and $k\in \N$, then for angles $\gamma, \theta, \tau$
\begin{align*}
\hspace{-0.9cm} P^{\alpha}_k(\cos\theta\cos\tau & +\sin\theta\sin \tau  \cos \gamma)\\
&=\sum_{i=0}^k c_{k,i}^\alpha \lrp{\sin\theta}^i \lrp{\sin\tau}^i P^{\alpha-\frac{1}{2}}_i(\cos\gamma) P^{\alpha+i}_{k-i}(\cos\theta) P^{\alpha+i}_{k-i}(\cos\tau),
\end{align*}
where $c_{k,i}^\alpha$ are some positive constants that depend on $\alpha,k,i$. More details about the formula can be found in~\cite{Addition}. Define
\begin{align*}
\cos\gamma  & = \tfrac{\ip{x}{y}_{[Z q]}}{\sqrt{\ip{x}{x}_{[Z q]}}\sqrt{\ip{y}{y}_{[Z q]}}}\\
\cos\theta  & = \tfrac{\ip{x}{q}_Z}{\sqrt{\ip{x}{x}_Z}\sqrt{\ip{q}{q}_Z}}, \ \sin\theta=\sqrt{1-\cos^2\theta},\\
 \cos\tau & =\tfrac{\ip{y}{q}_Z}{\sqrt{\ip{y}{y}_Z}\sqrt{\ip{q}{q}_Z}},\ \sin \tau =\sqrt{1-\cos^2\tau}.
\end{align*}
Now, expanding the expressions for the sines above and using the inverse formula for block matrices, one can show that
\begin{align*}
\sin\theta& =\sqrt{\tfrac{\ip{x}{x}_{[Z q]}}{\ip{x}{x}_Z}} \hspace{4.7cm}
\end{align*}
and
\begin{align*}
\cos\theta\cos\tau+\sin\theta\sin \tau \cos\gamma& =\tfrac{\ip{x}{y}_Z}{\sqrt{\ip{x}{x}_Z}\sqrt{\ip{y}{y}_Z}}. \hspace{7.8cm}
\end{align*}
Substituting the corresponding expressions and orders of Gegenbauer polynomials in the addition formula with $\alpha=\tfrac{n-r}{2}-1$, the result follows.
\end{proof}
\noindent The proof of Proposition~\ref{prop:addition} is based on the addition theorem for Gegenbauer polynomials~\cite{Addition} and is  related to the approach in~\citet{MultPsd}. Namely, \citet{MultPsd}  modifies the addition theorem for Gegenbauer polynomials to characterize p.d. kernels on $S^{n-1}$ invariant under the action of $\Stab_{O_n}(Z)$ for a given $Z\in (S^{n-1})^r$.

\bibliographystyle{plainnat}
\bibliography{references_kissing1}

\begin{thebibliography}{23}
\providecommand{\natexlab}[1]{#1}
\providecommand{\url}[1]{\texttt{#1}}
\expandafter\ifx\csname urlstyle\endcsname\relax
  \providecommand{\doi}[1]{doi: #1}\else
  \providecommand{\doi}{doi: \begingroup \urlstyle{rm}\Url}\fi

\bibitem[Andrews et~al.(1999)Andrews, Askey, and Roy]{Andrews}
G.E. Andrews, R.~Askey, and R.~Roy.
\newblock \emph{Special functions}.
\newblock Cambridge University Press, 1999.

\bibitem[Bachoc and Vallentin(2007)]{BVsphCodes}
C.~Bachoc and F.~Vallentin.
\newblock Semidefinite programming bounds for spherical codes.
\newblock In \emph{2007 IEEE International Symposium on Information Theory},
  pages 1801--1805, 2007.

\bibitem[Bachoc and Vallentin(2008)]{BV}
C.~Bachoc and F.~Vallentin.
\newblock New upper bounds for kissing numbers from semidefinite programming.
\newblock \emph{J. Amer. Math. Soc.}, 21\penalty0 (3):\penalty0 909--924, 2008.

\bibitem[Bachoc et~al.(2009)Bachoc, Nebe, de~Oliveira~Filho, and
  Vallentin]{LbChr}
C.~Bachoc, G.~Nebe, F.M. de~Oliveira~Filho, and F.~Vallentin.
\newblock Lower bounds for measurable chromatic numbers.
\newblock \emph{Geometric and Functional Analysis}, 19\penalty0 (3):\penalty0
  645--661, 2009.

\bibitem[Bochner(1941)]{FinPSD}
S.~Bochner.
\newblock Hilbert distances and positive definite functions.
\newblock \emph{Ann. of Math.}, 42:\penalty0 647--656, 1941.

\bibitem[de~Laat and Vallentin(2015)]{Packing}
D.~de~Laat and F.~Vallentin.
\newblock A semidefinite programming hierarchy for packing problems in discrete
  geometry.
\newblock \emph{Math. Program.}, 151\penalty0 (2):\penalty0 529--553, 2015.

\bibitem[de~Laat et~al.(2014)de~Laat, de~Oliveira~Filho, and
  Vallentin]{Packing2}
D.~de~Laat, F.~M. de~Oliveira~Filho, and F.~Vallentin.
\newblock Upper bounds for packings of spheres of several radii.
\newblock \emph{Forum of Mathematics}, Sigma 2, e23, 2014.

\bibitem[de~Laat et~al.(2018)de~Laat, Machado, de~Oliveira~Filho, and
  Vallentin]{kPointBounds}
D.~de~Laat, F.~Caluza Machado, F.~M. de~Oliveira~Filho, and F.~Vallentin.
\newblock $k$-point semidefinite programming bounds for equiangular lines.
\newblock \emph{{A}r{X}iv preprint}, 2018.
\newblock URL \url{https://arxiv.org/abs/1812.06045}.

\bibitem[DeCorte et~al.(2018)DeCorte, de~Oliveira~Filho, and
  Vallentin]{DistAvoid}
E.~DeCorte, F.~M. de~Oliveira~Filho, and F.~Vallentin.
\newblock Complete positivity and distance-avoiding sets.
\newblock \url{https://arxiv.org/abs/1804.09099}, 2018.

\bibitem[Delsarte et~al.(1977)Delsarte, Goethals, and Seidel]{DGS}
P.~Delsarte, J.M. Goethals, and J.J. Seidel.
\newblock Spherical codes and designs.
\newblock \emph{Geom. Dedicata}, 6:\penalty0 363--388, 1977.

\bibitem[Dobre et~al.(2016)Dobre, Dur, Frerick, and Vallentin]{alphaInf}
C.~Dobre, M.~Dur, L.~Frerick, and F.~Vallentin.
\newblock A copositive formulation for the stability number of infinite graphs.
\newblock \emph{Math. Program.}, 160\penalty0 (1):\penalty0 65--83, 2016.

\bibitem[Husem\"{o}ller(1994)]{Husemoller}
D.~Husem\"{o}ller.
\newblock \emph{Fibre Bundles}.
\newblock Springer-Verlag New York, 1994.

\bibitem[Kabatiansky and Levenshtein(1978)]{KabLev78}
A.~Kabatiansky and V.I. Levenshtein.
\newblock On bounds for packings on a~sphere and in space.
\newblock \emph{Problems of Information Transmission}, 14\penalty0
  (1):\penalty0 1--17, 1978.

\bibitem[Koornwinder(1977)]{Addition}
T.~Koornwinder.
\newblock Yet another proof of the addition formula for jacobi polynomials.
\newblock \emph{Journal of Mathematical Analysis and Applications}, 61\penalty0
  (1):\penalty0 136 -- 141, 1977.

\bibitem[Kuryatnikova and Vera(2018)]{Kissing1}
O.~Kuryatnikova and J.~C. Vera.
\newblock Positive semidefinite approximations to the cone of copositive
  kernels.
\newblock \emph{{A}r{X}iv preprint}, 2018.
\newblock URL \url{https://arxiv.org/abs/1812.00274}.

\bibitem[Machado and de~Oliveira~Filho(2017)]{BestBound}
F.~C. Machado and F.~M. de~Oliveira~Filho.
\newblock Improving the semidefinite programming bound for the kissing number
  by exploiting polynomial symmetry.
\newblock \emph{Experimental Mathematics}, 0:\penalty0 1--8, 2017.

\bibitem[Mittelmann and Vallentin(2010)]{MV}
H.~D. Mittelmann and F.~Vallentin.
\newblock High-accuracy semidefinite programming bounds for kissing numbers.
\newblock \emph{Experimental Mathematics}, 19\penalty0 (2):\penalty0 175--179,
  2010.

\bibitem[Musin(2014)]{MultPsd}
O.~R. Musin.
\newblock Multivariate positive definite functions on spheres.
\newblock In \emph{Proceedings of the AMS Special Session on Discrete Geometry
  and Algebraic Combinatorics}, pages 177--190, 2014.

\bibitem[Musin(2006)]{OneSidedKn}
O.R. Musin.
\newblock The one-sided kissing number in four dimensions.
\newblock \emph{Periodica Mathematica Hungarica}, 53\penalty0 (1):\penalty0
  209--225, Sep 2006.

\bibitem[Pfender(2007)]{Pfender}
F.~Pfender.
\newblock Improved {D}elsarte bounds for spherical codes in small dimensions.
\newblock \emph{Journal of Combinatorial Theory}, A 114\penalty0 (6):\penalty0
  1133--1147, 2007.

\bibitem[Schoenberg(1938)]{Schoenberg2}
I.~J. Schoenberg.
\newblock Metric spaces and positive definite functions.
\newblock \emph{Transactions of the American Mathemticl Society}, 44:\penalty0
  522--536, 1938.

\bibitem[Schoenberg(1942)]{Schoenberg}
I.~J. Schoenberg.
\newblock Positive definite functions on spheres.
\newblock \emph{Duke Math. J.}, 9\penalty0 (1):\penalty0 96--108, 1942.

\bibitem[Schrijver(2005)]{SchrijverTervAlg}
A.~Schrijver.
\newblock New code upper bounds from the terwilliger algebra and semidefinite
  programming.
\newblock \emph{IEEE Transactions on Information Theory}, 51\penalty0
  (8):\penalty0 2859--2866, 2005.

\end{thebibliography}

\end{document}